\documentclass[a4paper,11pt]{amsart}
\usepackage{amsmath,amssymb,amsfonts,amsthm,exscale,calc}
\usepackage[latin1]{inputenc}
\usepackage{latexsym,textcomp}
\usepackage{graphicx,epsf}
\usepackage{dsfont}
\usepackage[german, english]{babel}

\newtheorem{lemma}{Lemma}[section]
\newtheorem{theorem}[lemma]{Theorem}

\newtheorem{proposition}[lemma]{Proposition}
\newtheorem{corollary}[lemma]{Corollary}

\theoremstyle{definition}
\newtheorem{definition}[lemma]{Definition}
\newtheorem{example}[lemma]{Example}

\theoremstyle{remark}

\numberwithin{equation}{section}
\newcommand{\comment}[1]{}

\newcommand{\Z}{{\mathbb Z}}
\newcommand{\R}{{\mathbb R}}

\newcommand{\A}{{\mathcal{A}}}
\newcommand{\LL}{{\mathcal{L}}}

\newcommand{\QQ} {{\mathcal{Q}}}
\newcommand{\DD}{{\mathcal{D}}}
\newcommand{\FF}{{\mathcal{F}}}

\newcommand{\supp}{{\mathrm {supp}\,}}

\newcommand{\Deg}{{\mathrm {Deg}}}

\newcommand{\as}[1]{\left\langle #1\right\rangle}

\newcommand{\aV}[1]{\left\Vert #1\right\Vert}

\newcommand{\ow}[1]{\widetilde{ #1}}

\newcommand{\Hm}[1]{\leavevmode{\marginpar{\tiny%
$\hbox to 0mm{\hspace*{-0.5mm}$\leftarrow$\hss}%
\vcenter{\vrule depth 0.1mm height 0.1mm width \the\marginparwidth}%
\hbox to 0mm{\hss$\rightarrow$\hspace*{-0.5mm}}$\\\relax\raggedright
#1}}}

\def\id{\mathrm{id}}

\def\Deg{\mathrm{Deg}}

\renewcommand{\epsilon}{\varepsilon}
\renewcommand{\phi}{\varphi}

\begin{document}

\title[Diffusion determines the graph]{Diffusion determines the  recurrent graph}

\author[Keller]{Matthias Keller} 
\address{Mathematisches Institut  \\Friedrich Schiller Universit{\"a}t Jena \\07743 Jena, Germany } \email{m.keller@uni-jena.de}

\author[Lenz]{Daniel Lenz} \address{Mathematisches Institut \\Friedrich Schiller Universit{\"a}t Jena \\07743 Jena, Germany } \email{daniel.lenz@uni-jena.de}

\author[Schmidt]{Marcel Schmidt}\address{Mathematisches Institut \\Friedrich Schiller Universit{\"a}t Jena \\07743 Jena, Germany } \email{schmidt.marcel@uni-jena.de}

\author[Wirth]{Melchior Wirth}\address{Mathematisches Institut \\Friedrich Schiller Universit{\"a}t Jena \\07743 Jena, Germany } \email{melchior.wirth@uni-jena.de}

\date{\today}

\begin{abstract} We consider  diffusion on discrete measure spaces as  encoded by Markovian semigroups arising from weighted  graphs.  We study whether the graph is uniquely determined if the diffusion is given up to order isomorphism.
If the graph is recurrent then the complete   graph structure  and the measure space are  determined   (up to an overall scaling). As shown by counterexamples this result is optimal. Without the recurrence assumption,  the graph still turns out to be determined in the case of normalized diffusion on graphs with standard weights and in the case of arbitrary graphs over spaces in which each point has the same mass.

These investigations  provide discrete counterparts to studies of diffusion on Euclidean domains and manifolds initiated by Arendt and continued by Arendt / Biegert / ter Elst and Arendt / ter Elst.

A crucial step in our considerations shows that order isomorphisms are actually unitary maps (up to a scaling) in our context.
\end{abstract}

\maketitle

\tableofcontents

\section*{Introduction}
A famous question of M. Kac asks: \textit{Can one hear the shape of a drum?} \cite{Kac}.  There is a substantial amount of research concerning this question  over quite some time. Indeed,  from Milnor's counterexample in 16 dimensions \cite{Mil} over  Sunada's general method \cite{Sun} it took quite a while till  a  counterexample in two dimensions was given by Gordon / Webb /  Wolpert in \cite{GWW}. Of course, the question may also be addressed in a discrete setting, i.e., for graphs. However, there one can easily find counterexamples as has been known for quite some time, see e.g. the textbook \cite{CDS}.

In mathematical terms the question of Kac concerns the eigenvalues of the Laplacian on a surface and whether  a surface is determined by these eigenvalues. Thus, it asks whether  two surfaces with unitarily equivalent Laplacians are naturally congruent. Of course, two Laplacians are unitarily equivalent if and only if their unitary groups are unitarily equivalent. In this sense, the question of Kac deals with unitary groups and unitary equivalence.

Recently, Arendt \cite{Are}   studied  the question:  \textit{Does diffusion determine the body?}  Here, compared to the question of Kac, the unitary group is replaced by the diffusion semigroup and  unitary equivalence  is replaced by equivalence up to order  isomorphism. The question then  asks whether  two relatively compact domains in Euclidean space with  -- up to order isomorphism -- equal semigroups  are naturally congruent.
This question is then answered positively in \cite{Are} (under a weak regularity assumption on the domain). In fact, this can even be extended to  manifolds, as shown by Arendt / Biegert /  ter Elst in  \cite{ABtE}. For compact manifolds an alternative  treatment  was also given by Arendt / ter Elst   \cite{AtE}.

In the present paper we address this issue for diffusion on discrete measure spaces, i.e., for graphs. Our main result gives a positive answer to the question of Arendt in this setting. More specifically, our main result shows that

\begin{itemize}
\item diffusion determines the recurrent graph without killing and this result is optimal in a  certain sense (see Theorem~\ref{main_forms} and subsequent discussion).
\end{itemize}

When the underlying space is finite (corresponding to the compactness condition in \cite{Are,AtE}) the recurrence assumption is trivially satisfied and we obtain  an  analogue to the results of \cite{Are,AtE}.
In fact, this  can even be generalized to infinite graphs with finite  total edge weight. This is studied in Section~\ref{Normalized}. In particular,  we obtain  that

\begin{itemize}
\item diffusion determines  the  infinite graph provided the total edge weight is finite, Corollary~\ref{normalized-diffusion}.
\end{itemize}

In all these situations diffusion determines the graph in the sense
that it determines both the weights of the graphs and the measure on
the underlying space. Let us emphasize that the weights and the
measure are independent pieces of data. This is quite a difference
to the framework  of  \cite{Are,ABtE,AtE}, where  the measure is
canonically fixed by  the manifold structure.

Two most relevant situations in the graph setting actually come
with a  canonical measure. In these situations diffusion determines
the graph without any recurrence or compactness assumption.
Specifically, diffusion determines the graph

\begin{itemize}

\item  if the measure of each point is one, Theorem~\ref{theorem-m-equal-one}, or,

\item if the graph has standard weights (i.e., the edge weights take value in $\{0,1\}$) and the Laplacian in question is the normalized Laplacian, Theorem~\ref{unweighted}.
\end{itemize}
These results  results can be  considered as analogues to the
results of \cite{ABtE}.

Let us stress  that these are particularly important situations.
Indeed,  for a long time most studies  of Laplacians on graphs were
concerned with these situations.   In fact, it seems that the study
of more general situations  (as treated in the results  above) has
only become a focus of  attention during the last five years or so.
In this connection, we also point out some recent work on normalized
Laplacians such as \cite{BanJ,BHJ,BJ,BJL,HuJ}. We also take this
opportunity to mention the recent work \cite{HJ},  which introduces
and studies  discrete Laplacians in rather general geometric
situations.

\smallskip

Along our way we also show that, without any recurrence or compactness/ finiteness  assumption, the diffusion always determines

\begin{itemize}

\item the combinatorial structure of the graph and the combinatorial distance (Theorem~\ref{main_weak_form}),
    \item the intrinsic pseudo metric $\varrho$ introduced by Huang in \cite{Hua} (Theorem~\ref{rho-invariant}).
\end{itemize}
\medskip

While, obviously, our setting is quite different from \cite{Are,ABtE,AtE}, in terms of overall strategy our considerations certainly owe to \cite{Are,ABtE,AtE}. In this context  it seems worth pointing out some differences to these works.
  As for the results  one main difference was already mentioned above: In our situation the  weights and the  measure on the underlying space  are independent pieces of data (whereas the measure in the considerations of \cite{Are,ABtE,AtE} is actually canonically given). As for the method, we do  get some   additional insights  on the structure of intertwining order isomorphisms in our situation:
\begin{itemize}
\item Any intertwining order isomorphism is (up to a scaling) actually a unitary map (Corollary~\ref{U unitary}).
\item In the case of regular diffusions the  intertwining order isomorphisms are in one-to-one correspondence with generalized ground state transforms (Theorem~\ref{Gleichung} and Theorem~\ref{theorem_gst}).
\end{itemize}
The first point is rather remarkable as  the categories of unitary
maps and of order isomorphism are rather different in general. It
can be seen as the main insight in the present paper. Indeed, in one
way or other it is the crucial ingredient in the proofs of our
theorems. On a more abstract level it can  be understood as saying
that existence of an order isomorphism is indeed a stronger
requirement than existence of a unitary isomorphism and this
  can be seen in the context of explaining that order isomorphisms determine the graph while unitary isomorphisms do not.

\medskip

In order to provide the proper setting for our results we first
develop the theory of diffusion on discrete sets  in quite some
depth in Section~\ref{Graphs}. There, we basically follow
\cite{KL1,KL2} (see \cite{HK,HKLW} as well). These works provide a
framework for diffusion on discrete sets in terms of Dirichlet forms
and graphs.

We then turn to studying how the graph structure is determined via
diffusion. This is phrased in terms of order preserving maps
intertwining the corresponding semigroups. In this context we first
provide a general structure theorem on  order preserving maps
between $\ell^p$ spaces (Theorem~\ref{orderiso}) in
Section~\ref{Structure-oi}.  This result gives that any order
isomorphism arises as a composition of a bijection $\tau$  and a
scaling $h$. This  is a variant of a Lamberti type theorem. It  may
be of independent interest.
In Section~\ref{Structure}, we then proceed  to  study
those order preserving  isomorphisms which intertwine Laplacians on
graphs. Here, we obtain  a  formula relating the weights of the
different graphs in questions via the measures and the scaling $h$
and the bijection $\tau$ coming from the order isomorphism
(Theorem~\ref{Gleichung}) and  derive as a consequence  that  the
scaling $h$ is (super)harmonic (Corollary~\ref{h_harmonisch}). On
the technical level, these two results  are the main tools of the
paper. They are intimately linked to  the observation in
Corollary~\ref{U unitary}  that an intertwining order isomorphism is
unitary (up to a scaling).

The connection to generalized ground state transforms and
Theorem~\ref{theorem_gst} are then the content of
Section~\ref{Ground state section}.

As a first application, we can then show in Section~\ref{map} how
diffusion determines the graph structure and the pseudo metric
$\varrho$.

A final  ingredient in our reasoning is recurrence. This concept is
discussed  in Section~\ref{Recurrent}. The results of that section
are  well-known. We mainly include the discussion in order to be
self-contained.

Our main result, Theorem~\ref{main_forms} is then stated and proven
in Section~\ref{Diffusion}. The discussion following that theorem
shows the mentioned optimality of this result.

The particular case where the total edge weight  is finite is
considered in  Section~\ref{Normalized}. In Section~\ref{Discrete}
we consider two situation of special interest, viz the case where
the measure is normalized to be equal to one at each point and the
situation where the measure is given by the degree and the weights
take values in $\{0,1\}$. In these cases, the graph turns out to be
completely determined  even without a recurrence assumption.

\smallskip

The theory of local Dirichlet spaces, i.e., topological spaces
together with a strongly local Dirichlet form has been very
successful in generalizing the situation of Riemannian manifolds. We
consider our paper as a contribution to the emerging theory of
discrete Dirichlet spaces, i.e., discrete spaces together with a
Dirichlet form. For this reason we focus our attention on Dirichlet
forms on graphs and  the associated  semigroups in continuous time.
However, for graphs also  semigroups in discrete time  have been
considered. It turns out that essential part of our results remain
valid there as well. This is discussed in Section
\ref{Remark-discrete-time}. We thank the anonymous referee for
bringing up this point.

\smallskip

In order to make the paper self-contained we include  an  appendix
on the intertwining property.

\smallskip

Part of the material presented in this paper is based on the
Bachelor thesis of one of the authors (M. W.).

\medskip

\textbf{Acknowledgements.}  Partial support  by the German research foundation (DFG) is gratefully acknowledged.

\section{Graphs and diffusion on discrete measure spaces}\label{Graphs}
  This section deals with diffusion on discrete sets.  Here, diffusion is modeled by a Markovian semigroup.  Such  semigroups  are in one-to-one correspondence to  Dirichlet forms on the underlying space. Each such Dirichlet form gives naturally rise to a graph structure. In fact, regular Dirichlet forms are even in  one-to-one correspondence to graphs. Details are discussed below in this section.  More specifically, we first introduce the necessary background on graphs, then turn to associated Dirichlet forms and afterwards discuss the  associated Markovian semigroups. The  main thrust of the paper is to study whether  Markovian semigroups which are equivalent up to order isomorphism lead to the same  graph structure. Along our way, we will also investigate certain  metric features of the underlying graphs. The necessary background on metrics on graphs is discussed at the end of this section.
\medskip

Throughout, we let $X$ be a finite or  infinitely countable set. All functions on $X$ will be real valued.
\medskip

\textbf{Graphs.}  
A pair $(b,c)$ is called a
 \emph{graph} over $X$ if  $b:X\times X \to [0,\infty)$ is symmetric, has zero diagonal, and satisfies
\begin{align*}
    \sum_{y\in X}b(x,y)<\infty
\end{align*}
for all $x\in X$ and $c:X\to[0,\infty)$ is arbitrary. The elements of  $X$ are then referred to as \emph{vertices} and the function $b$ is called the  \emph{edge weight} and $c$ the \emph{killing term}. Elements  $x,y\in X$ are called \emph{neighbors} and said to be  \emph{connected by an edge of weight $b(x,y)$}, if $b(x,y)>0$. 
If the number of neighbors of each vertex is finite, then we call $(b,c)$ or $b$ \emph{locally finite}. A finite sequence $(x_{0},\ldots,x_{n})$ of pairwise distinct vertices such that $b(x_{i-1},x_{i})>0$ for $i=1,\ldots,n$ is called a \emph{path} from $x_0$ to $x_n$.
We say that the graph  $(b,c)$  is \emph{connected} if, for every two vertices $x,y\in X$, there is a path from $x$ to $y$.
\medskip


\textbf{Formal Laplacian and generalized forms.}  The space of real valued functions on $X$ is denoted by $C(X)$ and the subspace of functions with finite support is denoted by  $C_{c}(X)$.   Throughout the characteristic function of  a point $x\in X$ will be denoted as $1_x$, i.e.,
 $$\mbox{$1_x (y) = 1$ for $ x =y$ and $1_x (y) = 0$ otherwise.}$$

Given a graph $(b,c)$ over  $X$ we introduce the associated
\emph{formal Laplacian} $\LL =\LL_{b,c}$  acting on
\begin{align*}
\FF=\FF_{b,c}=\{f\in C(X) : \sum_{y\in X}b(x,y)|f(y)|<\infty\mbox{
for all }x\in X\}
\end{align*}
as
\begin{align*}
\LL f (x) = \sum_{y\in X}b(x,y)(f(x)-f(y))+ c(x ) f(x).
\end{align*}
The operator $\LL$  can be seen as a discrete analogue of  the Laplace Beltrami operator on a Riemannian manifold (and an additional potential).


Given a graph $(b,c)$ over  $X$, we furthermore  define the
\emph{generalized form} $\QQ = \QQ_{b,c}:
C(X)\to[0,\infty]$ by
\begin{align*}
\QQ(f):=\frac{1}{2}\sum_{x,y\in X}b(x,y)|f(x)-f(y)|^{2}+\sum_{x\in X}c(x)|f(x)|^{2}
\end{align*}
and  the \emph{generalized form domain} by
\begin{align*}
    \DD = \DD_{b,c} :=\{f\in C(X) :  \QQ(f)<\infty\}.
\end{align*}
Clearly,  $C_{c}(X)\subseteq \DD$ as $b(x,\cdot)$ is summable for every $x\in X$.

Since $\QQ^{1/2}$ is a seminorm and satisfies the parallelogram identity
\begin{align*}
\QQ(f+g)+\QQ(f-g)=2(\QQ(f)+\QQ(g)),\quad f,g\in\DD,
\end{align*}
it gives, by polarization, a semi scalar product on $\DD$ via
\begin{align*}
 \QQ(f,g)=\frac{1}{2}\sum_{x,y\in X} b(x,y) {(f(x)-f(y))}(g(x)-g(y))+\sum_{x\in X}c(x) {f(x)}g(x).
\end{align*}
In the case when $c\not\equiv 0$ and $b$ is  connected, the form $\QQ$ defines a scalar product.

Obviously, $\QQ$ is \textit{compatible with normal contractions} in the sense that
$$\QQ(C f)\leq \QQ (f)$$
holds for any $f\in C(X)$ and any normal contraction $C : \R \to \R$. (Here,  $C : \R \to \R $ is a \textit{normal contraction} if both $|C (p) |\leq |p|$ and $|C(p) - C(q)| \leq |p - q|$ hold for all $p,q\in \R$.)

In \cite{HK,HKLW} an `integration by parts'  was shown that allows one to pair functions in $\FF$ and $C_{c}(X)$ via $\QQ$. More precisely, the considerations of \cite{HK,HKLW} give  for $f\in \FF$ and $v\in C_{c}(X)$
\begin{align*}
 \frac{1}{2} \sum_{x,y\in X} & b(x,y) {(f(x)-f(y))}(v(x)-v(y))+\sum_{x\in X}c(x) {f(x)}v(x)\\
&=\sum_{x\in X} {f(x)}(\LL v)(x)\\
&= \sum_{x\in X} {(\LL  f)(x)} v(x),
\end{align*}
where all  sums converge absolutely. Moreover, it is shown there that  $\DD$ is a subset of $\FF$ and that for $f \in \DD$ the preceding sums all agree with $\QQ (f,g)$.  These formulae will be referred to as \textit{Greens formulae}.
\medskip

\textbf{Dirichlet forms and their  generators.}
We now assume that we are additionally given a measure $m$ on $X$ of full support.  Then, suitable restrictions of $\QQ$ are in correspondence with certain self-adjoint restrictions of the formal operator  $\ow L$ (defined below) on the Hilbert space $\ell^2 (X,m)$
of real valued  square summable functions equipped with the scalar product
\begin{align*}
   \as{u,v} =  \as{u,v}_m=\sum_{x\in X} {u(x)}v(x) m(x),
\end{align*}
and norm $\|u\|=\|u\|_m=\sqrt{\as{u,u}}$. This is discussed next:

Let $(b,c)$ be  graph over $X$ and let $\QQ = \QQ_{b,c}$ be the
associated form on the domain $\DD = \DD_{b,c}$ as discussed above.
Let now   $Q$ be a closed non-negative form on $\ell^2 (X,m)$, whose
domain $D(Q)$ satisfies $C_{c}(X)\subseteq D (Q) \subseteq \DD\cap
\ell^{2}(X,m)$ and which satisfies
\begin{align*}
Q(u,v)=\QQ(u,v),
\end{align*}
for $u,v\in D (Q)$. Thus, $Q$ is essentially a restriction of $\QQ$.
For such a form we define $Q(u):=Q(u,u)$ for $u\in D (Q)$ and
$Q(u):=\infty$ for $u\not\in D (Q)$. Such a form will be referred to
as a \textit{form associated to the graph $(b,c)$ over $(X,m)$}.  If
such a form is a Dirichlet form (i.e., satisfies $Q (C f) \leq Q (f)$
for all  $f\in D(Q))$ it will be called a \textit{Dirichlet form
associated to the graph $(b,c)$}.

By general results on closed forms, any  $Q$ associated to a graph comes with a unique self-adjoint operator $L$ such that $D (Q)$ is just the domain of definition of $L^{1/2}$ and $Q (f,g) = \langle L^{1/2} f, L^{1/2} g\rangle$ holds for all $f,g\in D(Q)$. The operator $L$ is known as the \textit{generator of the form $Q$}.
We will also  refer to  it as \textit{Laplacian corresponding to the form $Q$}.

By \cite[Proposition~3.3]{HKLW} we have that the Laplacian  $L$ corresponding to a form  $Q$ associated to a graph  satisfies
\begin{align*}
Lu=\ow Lu,
\end{align*}
for all $u\in D(L)$. Here, $\ow L$ is just the operator
 $ \frac{1}{m}  \LL$ acting as
\begin{align*}
\ow L f(x)=\frac{1}{m(x)}\sum_{y\in X}b(x,y)(f(x)-f(y))+\frac{c(x)}{m(x)}f(x).
\end{align*}


There are two natural examples, $Q^{(N)}$ and $Q^{(D)}$, which we refer to as the \emph{ forms with Neumann} and \emph{Dirichlet boundary conditions}, respectively.
The form $Q^{(N)}$ has the domain
\begin{align*}
D(Q^{(N)})=\DD\cap \ell^{2}(X,m)=\{u\in\ell^{2}(X,m) :  \QQ(u)<\infty\}.
\end{align*}
The form $Q^{(D)}$ has the domain
\begin{align*}
D(Q^{(D)})=\overline{C_{c}(X)}^{\aV{\cdot}_{ \QQ}},
\end{align*}
where
\begin{align*}
\aV{u}_{\QQ}:={(\QQ(u)+\aV{u}^{2})}^{\frac{1}{2}}.
\end{align*}
We denote the corresponding self adjoint operators by $L^{(N)}$ and $L^{(D)}$.

 By construction, both $Q^{(N)}$ and $Q^{(D)}$ are Dirichlet forms as $\QQ$ is  compatible with normal contractions. In fact, this is clear for $Q^{(N)}$ and follows for $Q^{(D)}$ by a result of Fukushima (see  \cite{KL1} for discussion).  Moreover, $Q^{(D)}$ is regular, i.e., $D(Q^{(D)})\cap C_{c}(X)$ is dense in $D(Q^{(D)})$ with respect to ${\aV{\cdot}}_{Q}$ and in $C_{c}(X)$ with respect to $\aV{\cdot}_{\infty}$.

The forms $Q^{(D)}$ and $Q^{(N)}$ are of particular relevance in the theory of forms associated to graphs  in the  sense that  a  symmetric, closed quadratic form $ Q$ is associated to the graph if and only if $Q^{(D)}\subseteq Q\subseteq Q^{(N)}$, see \cite{GHKLW}.
Here, we write $Q_{1}\subseteq Q_{2}$ if $D(Q_{1})\subseteq D(Q_{2})$ and $Q_{1}(u)=Q_{2}(u)$ for $u\in D(Q_{1})$.
\medskip

\textbf{Markovian semigroups.}
We will be interested in the semigroups associated to graphs. Here, we shortly sketch the necessary ingredients. For further discussion and proofs we refer to
standard references like \cite{FOT,Nag}.
\medskip

Any form $Q$ associated to a graph $(b,c)$ over $(X,m)$ is obviously  non-negative. Thus, the corresponding Laplacian $L$ gives rise to a semigroup $e^{-t L}$, $t\geq 0$. This semigroup provides the solution for the associated heat  equation as follows.   For any $u\in D(L)$ the function
$$\psi : [0,\infty)\to \ell^2(X,m),\,t\mapsto e^{-t L} u,$$ is a solution of the heat equation with initial value $u$
$$ (HE) \hspace{2cm}
\frac {d}{dt}\psi_t = \frac{1}{m} \LL  \psi_t, \;t>0\qquad
\psi_0 =u.$$

In fact, it is possible to describe the form $Q$  in terms of the associated semigroup as follows:

\begin{align*}
D(Q)&=\{f\in \ell^2 (X,m) :  \lim_{t\downarrow 0} \frac{1}{t} \langle f-e^{-tL}f,f\rangle <\infty\}\\
Q(f,g)&=\lim_{t\downarrow 0} \frac{1}{t} \langle f-e^{-tL}f,g\rangle.
\end{align*}
We will be interested in semigroups $e^{-t L}$, $t\geq 0$, which are
 \textit{Markovian}, i.e., satisfy   $ 0\leq e^{-t L} f\leq 1$ for all $t\geq 0$ whenever $f\in \ell^2(X,m)$ satisfies $0\leq f\leq 1$. Indeed, Markovian semigroups are the natural candidates for modeling diffusion processes.
By the second Beurling Deny criteria  the semigroup is Markovian if and only if the form $Q$ is a Dirichlet form.

\smallskip

If $e^{-t L}$ is a Markovian semigroup, then it leaves the space  $\ell^2 (X,m)\cap \ell^p (X,m)$  invariant and its restriction to this space can be extended uniquely to a strongly continuous  semigroup on $\ell^p (X,m)$ for any $p\in [1,\infty)$. The generator of this semigroup will be denoted by $L^{(p)}$ and the  semigroup will be denoted by $e^{-t L^{(p)}}$. It provides a  solution of the heat equation $(HE)$ with initial data $u\in \ell^p (X,m)$.

\smallskip

 The preceding discussion shows that graphs and the associated Dirichlet forms naturally encode diffusion on discrete sets.
\medskip

\textbf{Intrinsic metrics.} We finish this section with a discussion of  certain metrics on discrete sets arising from  graphs.  These metrics will turn out to be  invariant under the  order isomorphism considered later. Recall that a  \textit{pseudo metric} on the set $X$ satisfies all properties of a metric except that it may be degenerate (i.e., vanish outside the diagonal).  Subsequently we will consider metrics only for connected graphs. However,  the case of general graphs could easily be included e.g. by restricting attention to connected components and declaring the distance between different components to be infinity.
\medskip

\begin{definition}[Combinatorial metric]
Let  $(b,c)$ be a connected  graph  over $X$. Then, the metric $d$ defined  by
$$d(x,y)=\inf_\gamma \sharp \gamma,$$
 where the infimum is taken over all paths $\gamma=(x_0,\dots,x_n)$ connecting $x$ and $y$ and $\sharp \gamma =n$ is the number of edges in $\gamma$,  is called \emph{combinatorial graph metric}.
\end{definition}

The following definition of intrinsic metrics gives a class of (pseudo) metrics of graphs that allow one to study spectral geometry on graph in a similar way as on Riemannian manifolds. Such metrics were only  very recently brought forward and studied systematically for  arbitrary regular Dirichlet forms  in \cite{FLW} (see \cite{Uem,Fol} for related material as well). They have  already proven quite useful in the study of graphs \cite{BHK,BKW,GHKLW,GHM,HuK,Hua,Hua2,HKW,HKMW,KS}. While \cite{FLW} deals with general regular Dirichlet forms, here we only present the definition for  graphs.

\begin{definition}[Intrinsic metric for graphs \cite{FLW}]
Let $(b,c)$ be a graph over $(X,m)$. Then,  a pseudo metric $\delta$ on $X$ is said to be \emph{intrinsic} if
\begin{align*}
\sum_{y\in X}b(x,y) \delta(x,y)^2\leq m(x),\qquad x\in X
\end{align*}
\end{definition}

\textbf{Remarks.} A few remarks on this definition are in order. To simplify the discussion we only consider  the case $c=0$.

(a)  We  can write
\begin{eqnarray*}
\QQ (f,f) &= & \frac{1}{2} \sum_{x\in X}  \left(\frac{1}{ m(x)} \sum_{y\in X} b(x,y) |f(x) - f(y)|^2\right) m(x)\\
& =& \frac{1}{2} \sum_{x\in X} m(x) \|\nabla f (x)\|^2
\end{eqnarray*}
with
$$\|\nabla f(x)\| = \sqrt{\frac{1}{ m(x)} \sum_{y\in X} b(x,y) |f(x) - f(y)|^2.}$$
This suggests to single  out the set of functions $\mathcal{A}$ defined via $$\mathcal{A} :=\{f\in C(X) :  \|\nabla f(x)\| \leq 1 \:\mbox{for all $x\in X$} \}.$$
Then, a metric $\delta$ can easily be seen to be  intrinsic if and only if
$$\mbox{Lip}_\delta^1 \subset \mathcal{A},$$
where $\mbox{Lip}_\delta^1$ is the set of Lipschitz functions with constant one with respect to $\delta$.

\smallskip

(b) The set $\mathcal{A}$ is in general not closed under taking suprema: Consider  e.g. the subgraph of  $\Z$ given by  the integers from $-1$ to $1$, i.e.,
$$X=\{-1,0,1\},\quad m\equiv 1$$
 and
$$b(-1,0) = b(0,-1) = b(1,0) = b(0,1) = 1, \;\: b(1,-1) = b(-1,1) =0.$$
Let $f_+$ be the characteristic function of $1$ and $f_-$ the characteristic function of $-1$. Clearly, $f_\pm$ belong to $\mathcal{A}$.
On the other hand a direct calculation shows that $f= f_{+}+f_{-}=\max\{f_{+},f_{-}\}$ does not belong to $\mathcal{A}$ (see   an example in  \cite{FLW} for a similar reasoning as well).

\smallskip

(c) The set $\mbox{Lip}_\delta^1$ is closed under taking suprema by general principles. By (b),  the set $\mbox{Lip}_\delta^1$ will then  in general not be equal to   $\mathcal{A}$.  Put differently, in general
    $$\sigma (x,y):=\sup\{ |f(x) - f(y)|  : f\in \mathcal{A}\}$$
    will not be an intrinsic metric.  This is very different for strongly local Dirichlet forms; compare the Rademacher type theorem in \cite{FLW} and the considerations of \cite{Stol}.

\smallskip

(d) Define for  $\alpha >0$ the scaled version of the combinatorial metric via $d_\alpha = \alpha d$. Obviously,  $d_\alpha$ is then an intrinsic metric if and only if
    $$ \alpha  \sum_{y\in X} b(x,y) \leq m(x),$$
    for all $x\in X$.
On the other hand,  by a result of \cite{HKLW}, boundedness of the associated Laplacian is equivalent to existence of a $C>0$ such that
$$\frac{1}{C} \sum_{y\in X} b(x,y) \leq  m(x)$$
for all $x\in X$. Thus, (a suitably scaled version of) the  combinatorial metric $d$  is an intrinsic metric if and only if the associated Laplacian is bounded (compare  \cite{FLW,HKW} as well).
\medskip

A specific example of an intrinsic pseudo metric is the pseudo metric $\varrho$ introduced by  Huang in \cite{Hua}, Lemma 1.6.4., for connected graphs  and defined by
$$
\varrho : X\times X\to [0,\infty),
$$
$$\varrho(x,y)=\inf_\gamma \sum_{k=1}^n \min\left\{\frac 1{\sqrt{\Deg(x_{k-1})}},\,\frac 1{\sqrt{\Deg(x_{k})}}\right\}.
$$
Here, the infimum is taken over  all paths $\gamma=(x_0,\dots,x_n)$ connecting $x$ and $y$ and the \textit{generalized degree} $\Deg $ is the function
$$\Deg : X\to [0,\infty),\; x\mapsto \frac{1}{m(x)} \left(\sum_{y\in X} b(x,y) + c(y)\right).$$

\textbf{Remark.} In general $\varrho$ is not a metric. However, if the graph is locally finite then, clearly, any point has a positive distance to each of its neighbors and $\varrho$ is a metric.

\section{Structure of order isomorphisms between   $\ell^p$-spaces} \label{Structure-oi}
In this section we deal with  positive operators between
$\ell^p$-spaces. We present a structure theorem on order isomorphism
between $\ell^p$ spaces, Theorem~\ref{orderiso}.

\medskip

\begin{definition}[Order isomorphism] Let $(\Omega_1,\A_1,\mu_1),\,(\Omega_2,\A_2,\mu_2)$ be arbitrary measure spaces and $U \colon L^p(\Omega_1,\mu_1)\to L^p(\Omega_2,\mu_2)$ be linear. Then, $U$ is called \emph{order preserving} if $U f\geq 0$ whenever $f\geq 0$. The  map $U$ is called an \emph{order isomorphism} if $U$ is invertible and both $U$ and its inverse are order preserving.
\end{definition}

We start with an easy technical lemma, which is used in the proof of Theorem~\ref{orderiso}.
\begin{lemma}\label{product zero}
Let $(\Omega_1,\A_1,\mu_1),\,(\Omega_2,\A_2,\mu_2)$ be arbitrary measure spaces and $U\colon L^p(\Omega_1,\mu_1)\to L^p(\Omega_2,\mu_2)$ an order isomorphism. Let $f,g\in L^p(\Omega_1,\mu_1)$ be given such that $fg=0$. Then $(Uf)(Ug)=0$.
\end{lemma}
\begin{proof}
If $h\leq f,g$ then $Uh\leq Uf,Ug$ and $h'\leq h$ implies $Uh'\leq Uh$ since $U$ is an order isomorphism. We conclude $U(f\wedge g)=Uf\wedge Ug$ and $U(f\vee g)=Uf\vee Ug$. With $f^+=0\vee f$ and $f^-=-f\vee 0$ we obtain $U|f|=|Uf|$.\\
Furthermore, we have $fg=0$ if and only if $\mu_1$-almost everywhere $f(x)=0$ or $g(x)=0$ holds, that is, $|f(x)|\wedge |g(x)|=0$ for $\mu_1$-almost all $x\in \Omega_1$. Hence, $fg=0$ if and only if $|f|\wedge|g|=0$.\\
It follows that $|Uf|\wedge |Ug|=U(|f|\wedge|g|)=0$ and hence $(Uf)(Ug)=0$.
\end{proof}




The following theorem is a variant of the Banach-Lamberti type result (cf. \cite{Lam}, Theorem 3.1). Note that we work   with an order preserving transformation $U$ instead of isometries. This  gives us the  advantage that this theorem is valid also for $p=2$.

\begin{theorem}\label{orderiso}
Let $(X_1,m_1)$, $(X_2,m_2)$ be discrete measure spaces, $p\in[1,\infty)$ and $U\colon \ell^p(X_1,m_1)\to \ell^p(X_2,m_2)$ an order isomorphism. Then there exists a unique  function $h\colon X_2\to(0,\infty)$ and a unique  bijection $\tau\colon X_2\to X_1$ such that $Uf=h\cdot(f\circ\tau)$ for all $f\in \ell^p(X_1,m_1)$.
\end{theorem}
\begin{proof}
Uniqueness of $h$ and $\tau$ is clear. It remains to show existence.
For $y\in X_2$ define $S_y\colon \ell^p(X_1,m_1)\to\R,\,S_y(f)=Uf(y)$. Since $S_y$ is positive, it is a
continuous functional due to standard results on positive operators, see e.g. \cite[Theorem~12.3]{AB}.\\
By the representation theorem for the dual of $\ell^p$ there is a function $g_y\in \ell^q(X_1,m_1)$, $\frac 1p+\frac 1 q=1$, such that $S_y(f)=\sum_{\xi\in X_1}f(\xi)g_y(\xi)m_1(\xi)$ for all $f\in \ell^p(X_1,m_1)$.\\
Since $U$ is surjective, $\supp g_y$ is non-empty. Now let $x_0\in\supp g_y$ and $\,x\neq x_0$. Then we have $1_{x_0}1_x=0$ and by Lemma~\ref{product zero} we obtain
\begin{align*}
0=U1_x(y)\cdot U1_{x_0}(y)=g_y(x)m_1(x)g_y(x_0)m_1(x_0).
\end{align*}
Since $m_1$ has full support, we have $g_y(x)=0$. Thus, $\supp g_y=\{x_0\}$.\\
Hence, there is a function $h\colon X_2\to\R$ and a map $\tau\colon X_2\to X_1$ such that $Uf(y)=h(y)f(\tau(y))$ for all $f\in\ell^p(X_1,m_1)$, where $\supp g_y=\{\tau(y)\}$ and
\begin{align*}
h(y)=U1_{\tau(y)}(y)=g_y(\tau(y))m_1(\tau(y)).
\end{align*}
By positivity of $U$ we get $h(y)\geq0$ for all $y\in X_2$.\\
Analogously we obtain a function $\tilde h\colon X_1\to[0,\infty)$ and a map $\tilde\tau\colon X_1\to X_2$ such that $U^{-1}g=\tilde h\cdot g\circ \tilde \tau$ for all $g\in\ell^2(X_2,m_2)$. Hence, we have for all $y\in X_2$
\begin{eqnarray*}
1 &= & 1_y(y)=UU^{-1}1_y(y)\\
&= &  (h\cdot (\tilde h\cdot 1_y\circ\tilde \tau)\circ\tau)(y),\\
&= &\begin{cases}h(y)\tilde h(\tau(y))&\colon \tilde\tau(\tau(y))=y\\0&\colon\text{else},
\end{cases}
\end{eqnarray*}
that is, $\tilde\tau\circ\tau=\id_{X_2}$ and $h(y)\neq 0$ for all $y\in X_2$. The equation $\tau\circ\tilde\tau=\id_{X_1}$ is proven  analogously. Thus, $\tau$ is a bijection with inverse $\tilde\tau$.
\end{proof}

\textbf{Remark.}
If we think of elements of $\ell^p$ as sequences and operators between $\ell^p$-spaces as infinite matrices, the above theorem shows that the matrix of an order isomorphisms has exactly one non-zero entry in every row and column, which is strictly positive.

\begin{definition}[Bijection and scaling associated to order isomorphism] Let   $U\colon \ell^p(X_1,m_1)\to \ell^p(X_2,m_2)$ be an  order isomorphism  and $h\colon X_2\to(0,\infty)$ and  $\tau\colon X_2\to X_1$ the unique functions with  $Uf=h\cdot(f\circ\tau)$ for all $f\in \ell^p(X_1,m_1)$. Then, $\tau$ is called the \emph{bijection associated to} $U$ and $h$ is called \emph{the scaling associated} to $U$.
\end{definition}
For later use we now  compute an explicit formula for the adjoint of an order isomorphism.

\begin{lemma}\label{adjoint_orderiso} Let $p\in [1,\infty)$ be given and let $q $ with  $\frac{1}{p} + \frac{1}{q} = 1$ be chosen.
Let $\tau\colon X_2\to X_1$ be a bijection and $h\colon X_2\to (0,\infty)$ such that for all $f\in\ell^p(X_1,m_1)$ the function $h\cdot (f\circ\tau)$ belongs to $\ell^p(X_2,m_2)$. Define $U\colon \ell^p(X_1,m_1)\to \ell^p(X_2,m_2),\,f\mapsto h\cdot (f\circ\tau)$.
Then the adjoint of $U$ is given by
\begin{align*}
U^\ast\colon\ell^q(X_2,m_2)\to \ell^q(X_1,m_1),\,U^\ast f=\frac 1{m_1}(h\cdot f\cdot m_2)\circ\tau^{-1}.
\end{align*}
Moreover, for all $f\in C_c(X_1,m_1)$  and all $g \in C_c(X_2,m_2)$ we have
\begin{align*}
U^\ast Uf(x)&=\frac{m_2(\tau^{-1}(x))}{m_1(x)}h(\tau^{-1}(x))^2 f(x)\\
U U ^*g(y)&=\frac{m_2(y)}{m_1(\tau(y))}h(y)^2 g(y).
\end{align*}
\end{lemma}
\begin{proof}
The operator $U$ is obviously closed and defined on the whole space. Thus, it is a bounded operator. For this reason its adjoint is a bounded operator as well. Denoting the dual pairing between $\ell^p$ and $\ell^q$ by $(\cdot,\cdot)$ we can now compute for any $g\in \ell^q (X_2,m_2)$
\begin{eqnarray*}
U^\ast g (x) &=& (U^\ast g, \frac{1}{m_1 (x)}  1_x)\\
&=& (g , U ( \frac{1}{m_1 (x)} 1_x) )\\
&=& \sum_{y\in X_2} g(y) h (y)  \frac{1}{m_1 (x)} 1_x (\tau (y)) m_2 (y) \\
&=& g (\tau^{-1} (x)) h (\tau^{-1} (x)) \frac{m_2 (\tau^{-1} (x)) }{m_1(x)}.
\end{eqnarray*}

The second assertion follows easily:
\begin{align*}
U^\ast Uf(x)&=\frac 1{m_1(x)}(h\cdot Uf\cdot m_2)\circ\tau^{-1}(x)\\
&=\frac{((h^2\cdot m_2)\circ\tau^{-1})(x)} {m_1(x)}\cdot f(x).
\end{align*}
A similar computation for $UU^*$ finishes the proof.
\end{proof}

\section{Structure of order isomorphisms intertwining graphs}\label{Structure}
In Section~\ref{Graphs} we have seen that graphs naturally give rise to  Markovian semigroups. We are interested in  graphs with semigroups which are equal up to order isomorphism. This is captured in the concept of intertwining  order isomorphism.  In this section, we first discuss and characterize this  intertwining property and then study the functions $h$ and $\tau$ arising from intertwining order isomorphisms $U$. We will see that these functions behave well with respect to the corresponding edge weights. On the technical level this is the core of the paper.
\medskip

The following  proposition gives a characterization of intertwining of operators in terms of semigroups. It is certainly well known. For the convenience of the reader  we provide a proof in the appendix in  Proposition~\ref{intertwining_operators} (compare the  appendix of \cite{KLW} for related material as well).
\begin{proposition}
For $i=1,2$, let $Q_i$ be a Dirichlet form associated to the graph $(b_i,c_i)$ over  $(X_i,m_i)$. Let $p\in [1,\infty)$ be given and consider  the associated semigroups on $\ell^p (X_i,m_i)$ given  by $e^{-t L_i^{(p)} }$, $i = 1,2$.
Furthermore, let $U\colon \ell^p (X_1,m_1)\to \ell^p(X_2,m_2)$ be an order isomorphism.  Then the following assertions are equivalent:
\begin{itemize}
\item[(i)]$Ue^{-tL_1^{(p)} }=e^{-tL_2^{(p)} }U$ for all $t\geq 0$.
\item[(ii)]$U D(L_1^{(p)} )=D( L_2^{(p)} )$ and $U L_1^{(p)} f = L_2^{(p)}  Uf$ for all $f\in D( L_1^{(p)} )$.
\end{itemize}
\end{proposition}

\begin{definition}[Intertwining of operators]\label{def_intertwining_operators}
 Consider the situation of the proposition. The order isomorphism $U$ is said to \emph{intertwine} $L_1^{(p)} $ and $L_2^{(p)} $ if one of the equivalent assertions of the previous proposition holds.
\end{definition}

\begin{lemma}\label{lemma_properties_U}
For $i=1,2$, let $Q_i$ be a Dirichlet form associated to the graph $(b_i,c_i)$ over  $(X_i,m_i)$ and assume $(b_1,c_1)$ is connected. Let $p\in [1,\infty)$ be given and consider  the  semigroups associated to $Q_i$   on $\ell^p (X_i,m_i)$ given  by $e^{-t L_i^{(p)} }$ for  $i = 1,2$.
 Let $U\colon \ell^p (X_1,m_1)\to \ell^p(X_2,m_2)$ be an order isomorphism intertwining $L_1^{(p)} $ and $L_2^{(p)}$ with associated scaling $h$ and bijection $\tau$.  Then, the function
$$ X_1 \to (0,\infty),\, x \mapsto \frac{m_2(\tau^{-1}(x))}{m_1(x)}h(\tau^{-1}(x))^2$$
is constant.
\end{lemma}
\begin{proof} We denote the function $x \mapsto \frac{m_2(\tau^{-1}(x))}{m_1(x)}h(\tau^{-1}(x))^2$ by $\varphi$ and write $(\cdot,\cdot)$ for the dual pairing between $\ell^p$ and  $\ell^q$, with $\frac{1}{p} + \frac{1}{q} = 1$.  Lemma \ref{adjoint_orderiso} shows
\begin{align*}
\varphi(y) ( e^{-t L_1^{(p)} } 1_x,1_y ) &= (e^{-t L_1^{(p)} } 1_x,U^*U 1_y)\\&= ( U e^{-t L_1^{(p)} } 1_x,U 1_y )\\
&= ( e^{-t L_2^{(p)} } U1_x,U 1_y).
\end{align*}
As semigroups of Dirichlet forms act on all $\ell^p$ spaces, coincide on their intersections and are self-adjoint on $\ell^2$, the above implies
\begin{align*}
\varphi(y) \langle e^{-t L_1^{(2)} } 1_x,1_y\rangle &=\langle e^{-t L_2^{(2)} } U1_x,U 1_y\rangle\\
&= \langle e^{-t L_2^{(2)} }  U 1_y, U1_x \rangle\\
&= \varphi(x) \langle e^{-t L_1^{(2)} } 1_y, 1_x\rangle\\
&= \varphi(x) \langle e^{-t L_1^{(2)} }1_x,  1_y\rangle.
\end{align*}
Semigroups associated with connected graphs are positivity improving, that is $\langle e^{-t L_1^{(2)} } 1_x,1_y\rangle  > 0$ for all $x,y \in X_1$ (see e.g. Corollary 2.9 in \cite{KL1}). Dividing by this quantity finishes the proof.
\end{proof}

\begin{lemma} \label{intertwining in ell^2}
For $i=1,2$, let $Q_i$ be a Dirichlet form associated to the graph $(b_i,c_i)$ over  $(X_i,m_i)$ and assume $(b_1,c_1)$ is connected.
Let $p\in [1,\infty)$ be given and consider  the  semigroups associated to $Q_i$   on $\ell^p (X_i,m_i)$ given  by $e^{-t L_i^{(p)}}$ for  $i = 1,2$.
 Let $U\colon \ell^p (X_1,m_1)\to \ell^p(X_2,m_2)$ be an order isomorphism intertwining $L_1^{(p)} $ and $L_2^{(p)}$ with associated bijection $\tau$ and associated scaling $h$. Then there exists a constant $\beta >0$ such that the operator $\ow{U}:\ell^2(X_1,m_1) \to \ell^2(X_2,m_2)$, $f \mapsto h \cdot (f \circ \tau)$ satisfies
 $$\|f\|_{m_{1}}^2 = \beta \|\ow{U}f\|_{m_{2}}^2.$$
 Furthermore, $\ow{U}$ is an order isomorphism intertwining $L_1^{(2)}$ and $L_2^{(2)}$.
\end{lemma}

\begin{proof}
Since $\ow{U}$ and $U$ as well as $e^{-t L_i^{(p)}}$ and $e^{-tL_i^{(2)}}$ coincide on $\ell^2\cap \ell^p$, it suffices to show the continuity statement on $\ow{U}$ and its surjectivity.
Using Lemma~\ref{lemma_properties_U} we set
$$\beta = \frac{m_1(x)}{m_2(\tau^{-1}(x))h(\tau^{-1}(x))^2},$$
for some $x \in X_1$. Then an easy  direct computation shows
$$\|f\|_{m_{1}}^2 = \beta \|\ow{U}f\|_{m_{2}}^2,$$
 for each $f \in \ell^2(X_1,m_1)$. Thus, $\ow{U}$ is an isometry (up to a scaling).  Moreover, $\ow{U}$ obviously maps $C_c (X_1)$ onto $C_c (X_2)$. As $C_c$ is dense in $\ell^2$ and $\ow{U}$ is  an isometry (up to a scaling),  $\ow{U}$ is in fact   surjective. This finishes the proof.
\end{proof}

{\bf Remark.} Lemma~\ref{intertwining in ell^2} shows that intertwining on $\ell^p$ always implies intertwining on $\ell^2$. Hence, for a better readability of the statements we subsequently only formulate theorems for the case $p = 2$. The reader however should keep in mind that the theory remains true in the general case $p \in [1,\infty)$. We also drop the superscript on the $\ell^2$ generator of a Dirichlet form, i.e., write $L$ instead of $L^{(2)}$ (as already discussed in the first section).
\medskip

\begin{corollary} \label{U unitary} Assume the situation of the previous lemma.
 Then the operator $\sqrt{\beta}\ow{U}: \ell^2(X_1,m_1) \to \ell^2(X_2,m_2)$ is unitary.
\end{corollary}
\begin{proof}
By the previous lemma the operator $\sqrt{\beta}\ow{U}$ is a surjective isometry and hence unitary.
\end{proof}

\textbf{Remark.} The previous corollary provides a crucial ingredient of the investigations of this paper. Indeed, it is underlying the proof of the next theorem, which is the main technical tool in our considerations. On the abstract level it is remarkable in connecting the -- a priori rather different -- concepts of unitary isomorphisms and order isomorphisms in our context (compare discussion in the introduction).


\medskip

We can now come to the main technical result of the paper. Indeed, the subsequent theorem  gives the main connection between the weights of graphs whose Laplacians are intertwined by an order preserving isomorphism.

\begin{theorem}\label{Gleichung}
For $i=1,2$, let $Q_i$ be a Dirichlet form associated to the graph $(b_i,c_i)$ over  $(X_i,m_i)$ such that $(b_1,c_1)$ is connected. Let $U\colon \ell^2 (X_1,m_1)\to \ell^2(X_2,m_2)$ be an order isomorphism intertwining $L_1 $ and $L_2$ with associated bijection $\tau$ and scaling $h$. Then the Dirichlet forms satisfy  $D(Q_2) = UD(Q_1)$ and there exists some constant $\beta >0,$ such that
 $$Q_1(f,g) = \beta Q_2(Uf,Ug),$$
 for each $f,g \in D(Q_1)$. Furthermore, the equalities
\begin{align*}
m_1(\tau(w)) &= \beta h(w)^2 m_2(w)\\
b_1(\tau(x),\tau(y)) &= \beta h(x) h(y) b_2(x,y)\\
c_1(\tau(z)) &= \beta h(z) \mathcal{L}_2h(z),
\end{align*}
hold for all $w,x,y,z\in X_2$. Here $\mathcal{L}_2$ is the formal operator associated with $(b_2,c_2)$.
\end{theorem}
\begin{proof}
By Corollary~\ref{U unitary} there exists some $\beta >0$ such that the map $U$ satisfies
$$UU^* = \frac{1}{\beta}{\rm Id}_2 \text{ and }U^*U = \frac{1}{\beta} {\rm Id}_1,$$
where ${\rm Id}_i$ are the identity mappings on $\ell^2(X_i,m_i)$, $i=1,2$. Furthermore,  for this $\beta$ the proof of Lemma~\ref{intertwining in ell^2} shows
$m_1\circ \tau = \beta h^2 m_2$.

\smallskip

 We compute for $f \in D(Q_1)$
\begin{align*}
Q_1(f,f) &= \lim_{t\to 0} \frac{1}{t}\as{ f-e^{-tL_1}f,f}\\
&= \beta \lim_{t\to 0} \frac{1}{t}\as{ U^*Uf-e^{-tL_1}U^*Uf,f}\\
&= \beta \lim_{t\to 0} \frac{1}{t}\as{ Uf-e^{-tL_2}Uf,Uf}\\
&= \beta Q_2(Uf,Uf).
\end{align*}
This shows $Uf \in D(Q_2)$ and, therefore, $UD(Q_1) \subseteq D(Q_2)$. A similar computation with $U^{-1} = \beta U^*$ yields $UD(Q_1) = D(Q_2)$. The formula $ Q_1(f,g) = \beta Q_2(Uf,Ug)$ follows by polarization. For the statement on the graph structures we note that a short computation using Green's formula shows
$$Q_i(1_x,1_y) = \mathcal{L}_i 1_x (y) = \begin{cases} - b_i(x,y) , &\text{ if } x \neq y,\\
\sum_z b_i(x,z) + c_i(x), &\text{ if } x = y,\end{cases}$$
for $i = 1,2$. Thus, for $x,y\in X_2$ we may compute
\begin{align*}
b_1(\tau(x),\tau(y)) &= - Q_1(1_{\tau(x)},1_{\tau(y)})\\
 &= - \beta Q_2 (U1_{\tau(x)},U1_{\tau(y)}) \\
 &= - \beta h(x)h(y) Q_2(1_x,1_y)\\
 &= \beta h(x)h(y) b_2(x,y).
\end{align*}
Using this equality, we obtain for $z \in X_2$
\begin{align*}
c_1(\tau(z)) &= Q_1(1_{\tau(z)}) - \sum_{w \in X_1} b_1(\tau(z),w)\\
&= \beta  \left(h(z)^2 Q_2(1_z) - \sum_{v \in X_2} h(z)h(v) b_2(z,v) \right)\\
&= \beta h(z) \left(\sum_{v \in X_2} b(z,v)h(z) + c_2(z)h(z) - \sum_{v \in X_2} h(v) b_2(z,v) \right)\\
&= \beta h(z) \mathcal{L}_2 h (z).
\end{align*}
This finishes the proof.
\end{proof}
A relevant consequence  of the previous theorem is the following.
\begin{corollary}\label{degree}
For $i=1,2$, let $Q_i$ be a Dirichlet form associated to the graph $(b_i,c_i)$ over  $(X_i,m_i)$. Let $U\colon \ell^2 (X_1,m_1)\to \ell^2(X_2,m_2)$ be an order isomorphism intertwining $L_1 $ and $L_2$ with associated bijection $\tau$ and scaling $h$. Then
$${\rm Deg}_1(\tau(x)) = {\rm Deg}_2(x),$$
for all $x \in X_2$.
\end{corollary}
\begin{proof} Using the previous theorem we may compute
\begin{align*}
{\rm Deg}_1 (\tau(x)) &= \frac{1}{m_1(\tau(x))} \left(\sum_{y} b_1(\tau(x),y) + c_1(\tau(x)) \right) \\
&=  \frac{1}{m_1(\tau(x))} Q_1(1_{\tau(x)},1_{\tau(x)}) \\
&= \frac{\beta h(x)^2}{m_1(\tau(x))} Q_2(1_{x},1_{x}) \\
&= \frac{1}{m_2(x)} Q_2(1_{x},1_{x})\\
&= {\rm Deg}_2 (x). \qedhere
\end{align*}
\end{proof}

The statement of the previous theorem can be substantially strengthened if more information on the measures $m_1, m_2$ and / or the function $h$ is available. This is the content of later sections. Here, we derive next a special feature of $h$. We need one more definition.

\begin{definition}[(Super)Harmonic functions]
Let $(b,c)$  be a graph over $X$  with associated formal Laplacian $\LL$. A function $f\colon X\to\R$ is called \emph{superharmonic} if $f\in \mathcal{F}$ and $\LL f\geq 0$. It is called \emph{harmonic} if $f\in \mathcal{F}$ and $\LL f=0$.
\end{definition}

\textbf{Remark.} Note that a function $f$ is harmonic if and only if $\ow L  f = 0$ for any operator $\ow L$ of the form $\ow L   = \frac{1}{m} \LL$ with $m : X\to (0,\infty)$.

\smallskip

\begin{corollary}\label{h_harmonisch}
For $i=1,2$, let $Q_i$ be a Dirichlet form associated to the graph $(b_i,c_i)$ over  $(X_i,m_i)$.  Let $U\colon \ell^2 (X_1,m_1)\to \ell^2(X_2,m_2)$ be an order isomorphism intertwining $L_1 $ and $L_2$. Then the scaling  $h$  associated to $U$ is superharmonic. Furthermore, $h$ is harmonic if and only if $c_1 \equiv 0$.
\end{corollary}
\begin{proof}
As $h$ is strictly positive, this is an immediate consequence of the formula
\[c_1(\tau(x)) = \beta h(x) \mathcal{L}_2h(x). \qedhere\]
\end{proof}

 \textbf{Remark.} In this section, we have assumed connectedness of the graphs in some places. If this condition is  dropped the statements still  remain valid on connected components of the graph. Thus, all our (subsequent)  results will  hold true on connected components. In this sense,   connectedness can always be replaced with the statement being true on connected components.

\section{A generalized ground state transform} \label{Ground state section}

In this section we prove a converse of Theorem~\ref{Gleichung} for regular forms, which will show that our result is optimal if we do not impose further restrictions on the underlying graph structures. For this purpose we introduce the notion of generalized ground state transforms with respect to a superharmonic function.

Consider a graph $(b,c)$ over $(X,m)$ with associated regular Dirichlet form $Q^{(D)}$. Furthermore, suppose $h$ is a strictly positive superharmonic function, $\tau : X \to X$ is a bijection and $\beta >0$. We define the graph $(b_h,c_h)$ over $(X,m_h)$ by setting
\begin{align*}
m_h(\tau(w)) &= \beta h(w)^2 m(w),\\
b_h(\tau(x),\tau(y)) &= \beta h(x)h(y) b(x,y), \\
c_h(\tau(z)) &= \beta h(z) \mathcal{L}h (z),
\end{align*}
hold for all $w,x,y,z\in X$.
\begin{definition} The regular Dirichlet form $Q_h = Q_{h,\tau,\beta}$ associated with $(b_h,c_h)$ over $(X,m_h)$ is called \emph{generalized ground state transform} of $Q^{(D)}$ with respect to to triplet $(h,\tau,\beta)$. Its associated operator is denoted by $L_h$.
\end{definition}

{\bf Remark.} If $h$ is harmonic, $\tau \equiv$ id and $\beta \equiv 1$ the above construction is known as ground state transform of $Q^{(D)}$, see e.g. \cite{FLW,HK} for more discussions on the topic.

\begin{lemma} \label{Ground state transform} Let $(b,c)$ be a graph over $(X,m)$ and let $Q_{h,\tau,\beta}$ be a generalized ground state transform of the corresponding regular Dirichlet form $Q^{(D)}$. Then, the order isomorphism $U: \ell^2(X,m_h) \to \ell^2(X,m)$, $f \mapsto h \cdot (f\circ \tau)$ satisfies $ U D(Q_h) = D(Q^{(D)})$ and
$$Q_h(f,g) = \beta Q^{(D)} (Uf,Ug),$$
for each $f,g \in D(Q_h)$.
\end{lemma}
\begin{proof}
For $f \in \ell^2(X,m_h)$ the definition of $m_h$ yields
 $$\beta \|Uf \|_{m}^2 = \|f \|_{m_{h}}^2.$$
For $x \in X$ we can compute
\begin{align*}
Q_h(1_x,1_x) &= \sum_{y\in X} b_h(x,y) + c_h(x) \\
&= \beta \sum_{y\in X} h(\tau^{-1}(x))h(y)b(\tau^{-1}(x),y) + \beta h(\tau^{-1}(x))\mathcal{L} h(\tau^{-1}(x))\\
&= \beta h(\tau^{-1}(x))^2\sum_{y\in X} b(\tau^{-1}(x),y) + \beta h(\tau^{-1}(x))^2 c (\tau^{-1}(x))\\
&= \beta Q^{(D)}(U1_x,U1_x).
\end{align*}
For two distinct $x,y \in X$, we obtain
\begin{align*}
\beta Q(U1_x,U1_y) &=  - \beta h(\tau^{-1}(x)) h(\tau^{-1}(y)) b(\tau^{-1}(x),\tau^{-1}(y))\\&= -b_h(x,y) = Q_h(1_x,1_y).
\end{align*}
Now the bilinearity of $Q^{(D)}$ and $Q_h$ as well as the denseness of $C_c(X)$ in the form domains show the statement.
\end{proof}

\begin{theorem} \label{theorem_gst} Let $(b,c)$ be a graph over $(X,m)$. Let $Q^{(D)}$ be the corresponding regular Dirichlet form with associated operator $L^{(D)}$ and let $Q_{h,\tau,\beta}$ be a generalized ground state transform of $Q^{(D)}$. Then, the order isomorphism $U: \ell^2(X,m_h) \to \ell^2(X,m)$, $f \mapsto h \cdot (f\circ \tau)$ intertwines $L_h$ and $L^{(D)}$.
\end{theorem}

\begin{proof}
From the definitions it is obvious that $UU^* = \frac{1}{\beta} {\rm Id}$. Using the previous lemma we obtain for $f \in D(L_h)$ and $g \in D(Q^{(D)})$
\begin{align*}
\as{UL_h f,g} &= \as{L_hf, U^*g} \\
&= Q_h(f,U^*g) \\
&= \beta Q(Uf, UU^*g)\\
&= Q(Uf,g).
\end{align*}
This shows $Uf \in D(L^{(D)})$ and $L^{(D)}Uf = UL_hf$. A similar argument with $U^{-1}$ yields the inclusion $D(L^{(D)}) \subseteq D(L_h)$.
\end{proof}

{\bf Remark.} The previous theorem shows that the graph structure of a regular Dirichlet form is determined by its diffusion up to a generalized ground state transform. In particular, if the graph $(b,c)$ admits a strictly positive non-constant superharmonic function diffusion does not uniquely determine the graph.

\section{The map $\tau$ as an isometric graph isomorphism}\label{map}
In this section we use the material presented so far to obtain some results on how much  diffusion determines the graph in the general case. More specifically, the results below will show that graphs with Markovian semigroups which are equal  up to order isomorphism are isometric both with respect to combinatorial graph metric $d$ and to the intrinsic metric $\varrho$ defined via the generalized degree.
\medskip

\begin{theorem}[$\tau$ as isometry w.r.t. $d$] \label{main_weak_form}
For $i=1,2$, let $Q_i$ be a Dirichlet form associated to the graph $(b_i,c_i)$ over  $(X_i,m_i)$. Let $U\colon \ell^2 (X_1,m_1)\to \ell^2(X_2,m_2)$ be an order isomorphism intertwining $L_1 $ and $L_2$ with associated bijection $\tau$.
Then,  $b_1(\tau(y),\tau(y'))>0$ if and only if $b_2(y,y')>0$ holds. In particular, $(b_1,c_1)$ is connected if and only if $(b_2,c_2)$ is connected and   $\tau$ is an isometry with respect to the combinatorial graph metric in this case.
\end{theorem}
\begin{proof}
By Theorem~\ref{Gleichung} we have
\begin{align*}
b_1(\tau(y),\tau(z)) &= \beta h(y)h(z)  b_2(y,z)
\end{align*}
for all $y,z$ in the same connected component of $X_2$. Now our first claim follows because $h,\beta$ are strictly positive.\\
Moreover, the equation shows that $(y_0,\ldots,y_n)$ is a path in $(X_2,b_2,c_2)$ if and only if $(\tau(y_0),\ldots,\tau(y_n))$ is a path in $(X_1,b_1,c_1)$ which implies that $(X_1,b_1,c_1)$ is connected if and only if $(X_2,b_2,c_2)$ is connected as well as  $d_1(\tau(y),\tau(z))=d_2(y,z)$ for all $y,z\in X_2$.
\end{proof}

The above theorem shows that graphs with semigroups  that are equivalent via  an intertwining order isomorphism have the same combinatorial structure.

We can also  obtain the following variant of  Theorem~\ref{main_weak_form}.

\begin{theorem} [$\tau$ as  isometry w.r.t. $\varrho$]\label{rho-invariant}
For $i=1,2$, let $Q_i$ be a Dirichlet form associated to the connected graph $(b_i,c_i)$ over  $(X_i,m_i)$. Furthermore, let $U\colon \ell^2 (X_1,m_1)\to \ell^2(X_2,m_2)$ be an order isomorphism intertwining $L_1 $ and $L_2$ with associated bijection $\tau$.
Then $\tau$ is  isometric with respect to the pseudo metrics $\varrho_1,\varrho_2$.
\end{theorem}
\begin{proof}
This follows immediately from Corollary~\ref{degree} and
 Theorem~\ref{Gleichung} (compare proof of
 Theorem~\ref{main_weak_form} as well) since $\varrho_1$ and $\varrho_2$ only depend on the combinatorial structure (paths in the graph) and the vertex degrees.
\end{proof}

\section{Recurrent graphs}\label{Recurrent}
In this section we will study a class of graphs with the property
that every positive (super)harmonic function must be constant. This
will allow us to give a generalization of
Theorem~\ref{main_weak_form}. The results of this section are
essentially all known. For the convenience of the reader we include
some proofs.

\bigskip

\begin{definition}[Recurrence and transience]
A Markovian semigroup \\ $(e^{- t L} )_{t\geq 0}$ on $\ell^2(X,m)$ is called recurrent if
\begin{align*}
G(x,y):=\int_0^\infty e^{-t L}  1_x(y)\,dt=\infty
\end{align*}
for all $x,y\in X$.
It is called transient if $G(x,y)<\infty$ for all $x,y\in X$. A graph over $(X,m)$ is called recurrent/transient if the semigroup $(e^{-tL^{(D)}})_{t\geq 0}$ of the associated Dirichlet Laplacian $L^{(D)}$ is recurrent/transient.
\end{definition}

For a connected graph, the semigroup $(e^{-tL^{(D)}})_{t\geq 0}$
is irreducible (i.e., the semigroup maps non-negative, non-zero functions  to strictly positive ones) by \cite{HKLW} and, hence, the graph is either recurrent or transient \cite{FOT} (see \cite{Schm}, Proposition~3.2) as well.


The following lemma is rather easy to prove and  well-known in an
even more general context  \cite{FOT}. Thus, we refrain from giving
a proof.

\begin{lemma}
A Markovian semigroup $(e^{-t L})_{t\geq 0}$ on $\ell^2(X,m)$ is recurrent if and only if
\begin{align*}
(Gf)(y):=\int_0^\infty e^{-t L} f(y)\,dt=\infty
\end{align*}
for all $f\in\ell^2(X,m)$ with  $f\geq 0$ and $f\neq 0$ and all $y\in X$.
\end{lemma}

As a corollary we can easily obtain that recurrence is stable  under
order isomorphisms. Again, we omit the simple proof.

\begin{corollary} \label{stability_recurrence}
For $i=1,2$, let $Q_i$ be a Dirichlet form associated to the graph $(b_i,c_i)$ over  $(X_i,m_i)$ and let $L_i$ be the corresponding Laplacian.
Let $U: \ell^2(X_1,m_1)\to\ell^2(X_2,m_2)$ be an order isomorphism intertwining the Laplacians. Then, $(e^{-tL_1})_{t\geq 0}$ is recurrent if and only if $(e^{-tL_2})_{t\geq 0}$ is recurrent.
\end{corollary}


For later use we recall the following criterion for recurrence (see e.g. \cite{FOT,Soa}).

\begin{theorem}\label{characterization-Soardi} Let $(b,c)$ be a connected graph over $(X,m)$. Choose  $o\in X$ arbitrary and define the norm  $\|\cdot\|_o$ on $\DD$   by
    $$\|f\|_o = \sqrt{\QQ (f,f) + |f(o)|^2}.$$
    Then, the following assertions  are equivalent:

\begin{itemize}
\item[(i)] The graph is recurrent.

\item[(ii)] The closure of $C_c (X)$ with respect to $\|\cdot\|_o$   contains the constant function $1$ and $\QQ(1,1) = 0$.
\end{itemize}
\end{theorem}

\textbf{Remark.} Note that the theorem gives, in particular, that recurrence does not depend on the choice of the measure $m$.  Hence, from now on, we will call a graph recurrent/transient without referring to the measure $m$.
\medskip

The previous theorem yields the following lemma (cf. \cite{Schm}, Proposition~3.12) which justifies that we will assume $c=0$ in the following.
\begin{lemma}
If $(b,c)$ is a connected graph over $(X,m)$ with $c\neq 0$  then it is transient.
\end{lemma}
\medskip

From Theorem~\ref{characterization-Soardi}, we obtain the following
characterization of recurrence (see  \cite{Woe}, Theorem 6.21 as
well).

\begin{proposition}\label{kernel Laplacian recurrent}
Let be $(b,0)$ be  a connected graph over $(X,m)$.  Then the
following assertions are equivalent:

\begin{itemize}

\item[(i)] The graph is recurrent.

\item[(ii)] All positive superharmonic functions are constant.

\end{itemize}
\end{proposition}
\begin{proof} (i) $\Longrightarrow$ (ii): Let $\QQ$ be the generalized form associated to the graph $(b,0)$ and let $h$ be a positive superharmonic function. As $(b,0)$ is connected, $h$ is strictly positive (see e.g. Proposition 3.4 in \cite{HK}). We denote by $\QQ_h$ the generalized form corresponding to the graph of the ground state transform $(b_h,c_h)$ which was introduced in Section~\ref{Ground state section}.  Then, Lemma~\ref{Ground state transform}
yields
$$\QQ_h \left(\frac{\varphi}{h},\frac{\varphi}{h}\right) = \QQ (\varphi,\varphi)$$
for all $\varphi \in C_c (X)$.  By Theorem~\ref{characterization-Soardi}, we can now find a sequence $(\varphi_n)$ in $C_c (X)$ converging to $1$ with respect to $\|\cdot\|_o$. As discussed in \cite{Soa} (see e.g. \cite{GHKLW} as well) convergence with respect to $\|\cdot\|_o$ implies pointwise convergence. Thus, we can apply the Fatou lemma to the above equality between $\QQ_u$ and $\QQ$ to obtain that
$$0 \leq \QQ_h \left(\frac{1}{h},\frac{1}{h}\right) = \QQ (1,1)  =0.$$
By connectedness of the graph, this  gives that $\frac{1}{h}$  is constant. Hence, $h$ is constant.

\smallskip

(ii)$\Longrightarrow$ (i):  Assume $(b,0)$ is not recurrent. We let $\DD_o = \overline{C_c(X)}^{\|\cdot\|_o}$. By standard Hilbert space arguments there exists a unique minimizer $h$ of the quadratic functional $\|\cdot \|^2_o$ on the closed convex set $\{u \in \DD_o: u(o) \geq 1\}.$ As $\|(0\vee h) \wedge1\|_o \leq \|h\|_o$, the function $h$ is positive and satisfies $h(o) = 1$. We now show that $h$ is superharmonic. For $x\in X$ and $\varepsilon > 0$, we obtain
\begin{align*}
\QQ(h,h) &= \|h\|_o - 1 \\&\leq \|(h + \varepsilon 1_x) \wedge1 \|_o - 1\\
 &= \QQ((h + \varepsilon 1_x) \wedge1 ,(h + \varepsilon1_x) \wedge1 )\\ &\leq \QQ(h + \varepsilon 1_x,h + \varepsilon 1_x)\\
 &= \QQ(h,h) + 2 \varepsilon \mathcal{L}h(x) + \varepsilon^2 \QQ(1_x,1_x).
\end{align*}
As this holds for all $\varepsilon >0$, we infer $\mathcal{L} h (x)
\geq 0$ for all $x\in X$. As the graph is not recurrent the constant
functions do not belong to $\DD_o$, hence, $h$ is not constant.
\end{proof}

Finally, we note the following (see \cite{Kuw}, Theorem 6.2 as well).

\begin{proposition}\label{regular} If the graph  $(b,0)$ over $(X,m)$ is recurrent, then the associated Dirichlet forms $Q^{(D)}$ and $Q^{(N)}$ agree.  In particular, there is then only one Dirichlet form associated to the graph.
\end{proposition}

\begin{proof} Assume $(b,0)$ is recurrent. Let $(\varphi_n)$ be a sequence in $C_c(X)$ converging to $1$ with respect to $\|\cdot\|_o$. By the cut-off properties of $\QQ$ this sequence can be chosen to satisfy $0 \leq \varphi_n \leq 1$. Let $f \in D(Q^{(N)})$ be bounded. Then $ \psi_n := f \cdot \varphi_n \in D(Q^{(D)})$ for each $n$. Furthermore, as $Q^{(N)}$ is a Dirichlet form, we obtain from Theorem 1.4.2. of \cite{FOT}
$$Q^{(D)}(\psi_n,\psi_n) = Q^{(N)}(\psi_n,\psi_n) \leq \|f\|_\infty \sqrt{Q^{(N)}(\varphi_n)} + \|\varphi_n\|_\infty \sqrt{Q^{(N)}(f)}. $$
This shows that $(\psi_n)$ is  bounded with respect to the form norm $\|\cdot\|_{Q}.$ Thus, $(\psi_n)$ has a weakly convergent subsequence with limit $\psi \in D(Q^{(D)})$. As $\psi_n \to f$ pointwise, we infer $f = \psi$. Every function in $D(Q^{(N)})$ can be approximated by bounded functions, hence, the claim follows.
\end{proof}

The previous proposition allows us to speak about \textit{the} Dirichlet form associated to $(b,0)$ over $(X,m)$ whenever the graph $(b,0)$ is recurrent.

\section{Diffusion determines the recurrent graph}\label{Diffusion}
In this section we show that a recurrent graph is completely determined by the  diffusion associated to it (up to an overall constant). As we will see the condition of recurrence  is necessary and the appearance of the overall constant can not be avoided.
\medskip

Recall that recurrence implies the uniqueness of the associated Dirichlet form.
\begin{theorem}\label{main_forms}
Let $(b_1,0)$ be a connected recurrent graph over $(X_1,m_1)$ and $Q_1$ the associated Dirichlet form and $L_1$ its generator.  Let $Q_2$ be a Dirichlet form associated to a graph $(b_2,0)$ over  $(X_2,m_2)$ and $L_2$ its generator. Let $U\colon \ell^2 (X_1,m_1)\to \ell^2(X_2,m_2)$ be an order isomorphism intertwining $L_1 $ and $L_2$ and  $\tau$ the associated bijection and $h$ the associated scaling. Then  there is  a constant $\beta>0$ such that
\begin{align*}
b_1(\tau(y),\tau(z))&=\beta b_2(y,z),\\
m_1(\tau(y))&=\beta m_2(y)
\end{align*}
hold for all $y,z\in X_2$.
\end{theorem}
\begin{proof}
By Lemma~\ref{h_harmonisch}, the function $h$ is harmonic. As $(b_2,0)$ is also recurrent by Corollary \ref{stability_recurrence},  the function $h$ is constant by Proposition~\ref{kernel Laplacian recurrent}  and the statement follows from Theorem~\ref{Gleichung}.
\end{proof}
\medskip

{\bf Remark.} The theorem is optimal in the following sense:
The condition of recurrence (and hence of  vanishing $c$) is necessary. If a graph is not recurrent then, as seen in the previous section, it admits a positive non-constant superharmonic function. As shown by the discussion in Section~\ref{Ground state section} such superharmonic functions lead to intertwined operators on different graph structures. Similarly, the appearance of the constant $\beta$ can not be avoided in the general case  as one could do a generalized ground state transform with the constant function $1$ and some normalizing constant $\beta >0$.
\medskip

From the theorem we also obtain the following characterization.

\begin{corollary}
For $i=1,2$, let $Q_i$ be a Dirichlet form associated to the graph $(b_i,0)$ over  $(X_i,m_i)$ and let $L_i$ be the corresponding Laplacian. Assume that both graphs are recurrent and connected. Then, the following assertions are equivalent:
\begin{itemize}
\item[(i)]There is an order isomorphism $U\colon \ell^2(X_1,m_1)\to \ell^2(X_2,m_2)$ intertwining $L_1$ and $L_2$.
\item[(ii)]There is a bijection $\tau\colon X_2\to X_1$ and a constant $\beta>0$ such that
\begin{align*}
b_1(\tau(y),\tau(z))&=\beta b_2(y,z)\\
m_1(\tau(y))&=\beta m_2(y)
\end{align*}
for all $y,z\in X_2$ and $L_2$ is the Dirichlet (resp. Neumann) Laplacian on the graph $(b_2,0)$ over $(X_2,m_2)$.
\end{itemize}
\end{corollary}
\begin{proof}
(i)$\Rightarrow$ (ii) is a consequence of the previous theorem.  (ii)$\Rightarrow$(i) follows from the discussion about the generalized ground state transform, Theorem~\ref{theorem_gst} and the fact that recurrence implies the regularity of $Q_i$, $i = 1,2$, Proposition~\ref{regular}.
\end{proof}

\textbf{Remark.} Of course, if $X$ is finite and $c$ vanishes the recurrence assumption is automatically satisfied. This can be substantially strengthened as discussed in the next section.


\section{Application to graphs with finite total  edge weight}\label{Normalized}
As noted in the previous section graphs over finite sets are determined by their diffusion. Of course, any measure on a finite set is finite and the Laplacian associated to a graph over a finite set  is bounded. It turns out that these two requirements alone already imply that the diffusion determines the graph.  In fact, these two requirements imply finiteness of the total edge weight, which in turn implies recurrence.   As a consequence  we obtain that the  diffusion determines the graph whenever the total edge weight is finite, Corollary~\ref{normalized-diffusion}.  These results can be considered as rather complete  analogues of \cite{Are,AtE} in the sense that the (relative) compactness assumption there is replaced by finiteness of the measure and boundedness of the Laplacian.
\medskip

We start with a characterization of boundedness of the Laplacian taken from \cite{HKLW}.

Recall that the \emph{generalized degree} $\Deg:X\to[0,\infty)$ of a graph $(b,c)$ over $(X,m)$ is defined by
$$\Deg (x) := \frac{1}{m(x)} \left( \sum_{y\in X} b(x,y) + c(x) \right),\quad x\in X.$$
\begin{proposition}[Characterization of boundedness \cite{HKLW}]\label{p:bounded}  Let $(b,c)$ be a graph over $(X,m)$. Then, the following assertions are equivalent:
\begin{itemize}
\item[(i)] There exists a $C>0$ such that $\Deg (x) \leq C $
holds for all $x\in X$.
\item[(ii)]  There exists a $c>0$ such that
$\QQ (f,f)\leq c \|f\|^2$
for all $f\in \ell^2 (X,m)$.
\item[(iii)] The operator  $\ow L$  restricts to a bounded operator on  $\ell^2 (X,m)$.
\item[(iv)] The operator $\ow L$ restricts to  a bounded operator on  $\ell^p (X,m)$ for all $p\in [1,\infty]$.
\end{itemize}
\end{proposition}

\begin{definition} In the situation of the proposition we say that the graph $(b,c)$ over $(X,m)$ has a \emph{bounded Laplacian}.
\end{definition}

Furthermore, for a graph $(b,c)$ over $X$ recall the \emph{normalizing measure} $n:X\to(0,\infty)$ given by
$$n(x) = \sum_{y\in X} b(x,y)+c(x),\quad x\in X.$$

\textbf{Example.} (Normalized Laplacian and normalized form) A special instance of a bounded Laplacian is the \textit{normalized Laplacian}  associated to the graph $(b,0)$ over $X$. This Laplacian arises by the choice of the measure $m=n$ such that the generalized degree defined above equals one. This normalized Laplacian is heavily studied (see discussion in the introduction). We will refer to the induced Markovian semigroup as normalized diffusion on the graph $(b,0)$.
\medskip

In the case of finite $X$, the associated Laplacian is bounded and the total mass is finite. It is possible to characterize the occurrence of these two features for general countable $X$  via the \textit{total edge weight}
$$n(X)=\sum_{x,y\in X} b(x,y) + \sum_{x\in X} c(x).$$
This is the content of the next proposition.

\begin{proposition} Let $(b,c)$ be a graph over $X$. Then, the following two assertions are equivalent:
\begin{itemize}

\item[(i)] There exists a measure $m$ on $X$ with finite total mass such that the graph $(b,c)$ over $(X,m)$ has a bounded Laplacian.

\item[(ii)] The total  edge weight
$$n(X)=\sum_{x,y\in X} b(x,y) + \sum_{x\in X} c(x)$$
is finite.
\end{itemize}
\end{proposition}
\begin{proof} (i)$\Longrightarrow$ (ii):  Let $m$ be such a measure. By the first proposition of this section, there exists then a $C\geq 0$ with
$\Deg   \leq C$. A  short computation now  gives
$$\sum_{x,y\in X} b (x,y) + \sum_{x\in X} c(x)  =\sum_{x\in X} \Deg (x)m(x) \leq C m(X) < \infty.$$

\smallskip

(ii)$\Longrightarrow$ (i): Consider the measure $m = n$. Then, $m$ has finite total mass by (ii) and the associated Laplacian is bounded.
\end{proof}

As is certainly well-known any graph with finite total edge weight is recurrent. We include a proof for completeness reasons.

\begin{proposition} Let $(b,0)$ be a connected graph over $X$ with finite total edge weight. Then, the graph is recurrent.
\end{proposition}
\begin{proof} This follows easily from Theorem~\ref{characterization-Soardi}. In fact, any sequence $(\varphi_n)$ in $C_c (X)$ with $0 \leq \varphi_n \leq 2$ and $\varphi_n (x) \to 1$, $n\to \infty$,  for all $x\in X$ will converge to $1$ with respect to $\|\cdot\|_o$ due to the finiteness assumption on the total edge weight.
\end{proof}




From the previous proposition and Theorem~\ref{main_forms} we immediately obtain the following corollary.

\begin{corollary}[Diffusion determines the graph of total finite edge weight] \label{normalized-diffusion} Let $(b_1,0)$ be a connected  graph over $(X_1,m_1)$ with
$$\sum_{x,y} b_1 (x,y) < \infty.$$
  Let $Q_1$ be the associated Dirichlet form with associated Laplacian $L_1$.  Let $Q_2$ be a Dirichlet form associated to the graph  $ (b_2,0)$ over $(X_2,m_2)$ and let $L_2$ be its generator.
 Let $U: \ell^2 (X_1,m_1)\to \ell^2(X_2,m_2)$ be an order isomorphism intertwining $L_1 $ and $L_2$ with associated bijection $\tau$ and associated scaling $h$.
  Then  there is  a constant $\beta>0$ such that
\begin{align*}
b_1(\tau(y),\tau(z))&=\beta b_2(y,z),\\
m_1(\tau(y))&=\beta m_2(y)
\end{align*}
hold for all $y,z\in X_2$.
\end{corollary}
\medskip

 \textbf{Remark.} As discussed in the introduction to this subsection, the previous theorem generalizes the case of finite $X$ and can be seen to provide an analogue to the results of \cite{Are,AtE}.


\section{Two  special situations }\label{Discrete}
In this section, we have a look at two special  situations, where  we can say more even without a recurrence condition. These two situations are the case of constant measure and the case of a graph with standard weights and normalized Laplacian (see discussion in the introduction).
\medskip

\textbf{The case $m = 1$.} If the measure is normalized to be equal to one in each point, we obtain the following theorem.

\begin{theorem}\label{theorem-m-equal-one}
For $i=1,2$, let $Q_i$ be a Dirichlet form associated to the graph $(b_i,c_i)$ over  $(X_i,m_i)$.  Assume that $(b_1,c_1)$ is connected and $m_1\equiv1$ and $m_2 \equiv 1$. Let $U: \ell^2 (X_1,m_1)\to \ell^2(X_2,m_2)$ be an order isomorphism intertwining $L_1 $ and $L_2$ with associated bijection $\tau$.
 Then the equalities  $b_1(\tau(y),\tau(z))=b_2(y,z)$ and $c_1(\tau(y))=c_2(y)$ hold for all $y,z\in X_2$.
\end{theorem}
\begin{proof}
By Theorem~\ref{Gleichung} there exists a constant $\beta > 0$ such that $$m_1( \tau(x)) = \beta h(x)^2 m_2(x).$$
Using $m_1 \equiv 1$ and $m_2 \equiv 1$ we obtain $h(x) = \frac{1}{\sqrt{\beta}}$ for all $x \in X_2$. Now the statement follows from the other formulas of Theorem~\ref{Gleichung}.
\end{proof}

\textbf{Remark.}
As the proof shows, it suffices to assume that $\tau$ is measure preserving, that is, $m_2(y)=m_1(\tau(y))$ for all $y\in X_2$.
\medskip

\textbf{The normalized Laplacian for graphs with standard weights.}
Here, we consider the situation that $(b,0)$ is a graph over $X$
with $b$ taking values in $\{0,1\}$  and the measure $n$ given by $n(x) =
\sum_{y} b(x,y)$. Thus, we consider the normalized Laplacian on a
graph with standard weights.  In this situation, the associated
Dirichlet form is bounded, Proposition~\ref{p:bounded}. In particular, it is the unique form
associated to the graph $(b,0)$ over $(X,n)$. From
Theorem~\ref{main_weak_form} we immediately infer the following
result.

\begin{theorem}\label{unweighted}
For $i = 1,2$,  let $(b_i,0)$ be a graph over $X_i$ with $b_i$ taking values in $\{0,1\}$ and define the measure $n_i$ on $X_i$ by $n_i (x) = \sum_{y\in X_i} b_i (x,y)$ and let $Q_i$ be the associated  Dirichlet form. Suppose $U : \ell^2 (X_1,n_1)\to \ell^2(X_2,n_2)$ is an order isomorphism intertwining $L_1 $ and $L_2$ with associated bijection $\tau$.
Then, the equality
$b_1(\tau(y),\tau(z))=b_2(y,z)$ holds   for all $y,z\in X_2$.
\end{theorem}
\begin{proof} As both $b_1$ and $b_2$ take values only in $\{0,1\}$ the equality $$b_1(\tau(y),\tau(z))=b_2(y,z),$$
 for $y,z\in X_2$  follows directly from Theorem~\ref{main_weak_form}.
\end{proof}

\textbf{Remark.} In the situation of the previous Theorem it is even
possible to include non-vanishing $c$'s in the graphs. One just has
to convince oneself that one still has unique forms associated to
the arising graphs $(b_i,c_i)$  due to the boundedness of the forms
induced by $(b_i,0)$, $i=1,2$.

\section{Remark on the case of discrete time}\label{Remark-discrete-time}
Instead of studying intertwining of continuous time semigroups which
describe the time evolution of a continuous time Markov chain we can
also consider their discrete time counterpart. The raised question
then becomes: 'Do discrete time Markov chains determine the graph?'
This is briefly discussed in this section.

\bigskip

For a graph $(b,0)$ over the vertex set $X$, we consider the {normalizing measure} $n : X \to (0,\infty)$ given by
$$n(x)  = \sum_{y \in X} b(x,y) $$
and the \emph{Markov operator} $P: \mathcal{F} \to C(X)$ by setting
$$Pf(x) = \frac{1}{n(x)} \sum_{y \in V} b(x,y) f(y).$$
The choice of the measure $n$ ensures that $P$ is a bounded self-adjoint operator when restricted to $\ell^2(X,n).$ The powers of the Markov operator form a discrete time semigroup which describes the time evolution of the Markov chain $(X_n)_{n\geq 1}$ with transition probabilities
$$\mathbb{P}(X_n = x \, | \, X_{n-1}  = y) = \frac{b(x,y)}{n(x)}.$$
Note that all powers of two Markov operators are intertwined by an order isomorphism if and only if the Markov operators itself are intertwined. Thus, the analogue of Theorem \ref{Gleichung} in discrete time reads as follows.

\begin{theorem}\label{Gleichung_discrete}
 For $i = 1,2$,  let $(b_i,0)$ be a graph over $X_i$ with corresponding Markov operator $P_i$ and let $(b_1,0)$ be connected. Furthermore, let $U:\ell^2(X_1,n_1) \to \ell^2(X_2,n_2)$ be an order isomorphism with associated bijection $\tau$ and scaling $h$, such that
 $U P_1 = P_2 U. $
 Then there exists a constant $\beta >0$, such that the equalities
\begin{align*}
n_1(\tau(w)) &= \beta h(w)^2 n_2(w),\\
b_1(\tau(x),\tau(y)) &= \beta h(x) h(y) b_2(x,y),\\
P_2h(z) &= h(z),
\end{align*}
hold for all $w,x,y,z \in X_2$.
\end{theorem}

\begin{proof}
Let $Q^{(D)}_i$ be the regular Dirichlet form associated with the graph $(b_i,0)$ over $(X_i,n_i)$. Then, the associated operator $L_i$ is given by
$$L_i f(x)  =  \frac{1}{n_i(x)} \sum_{y \in X_i} b_i(x,y) (f(x) - f(y)) = f(x) - P_if(x).$$
This shows that the  $L_i$ are bounded and satisfy $UL_1 = L_2 U$. Thus, the statement follows from Theorem~\ref{Gleichung} and the observation that $P_2 h = h$ is equivalent to $\mathcal{L}_2 h = 0$.
\end{proof}

As in the continuous time setting we can strengthen this theorem for recurrent graphs.

\begin{theorem}
 For $i = 1,2$,  let $(b_i,0)$ be a graph over $X_i$ with corresponding Markov operator $P_i$ and let $(b_1,0)$ be recurrent and connected. Furthermore, let $U:\ell^2(X_1,n_1) \to \ell^2(X_2,n_2)$ be an order isomorphism with associated bijection $\tau$ and scaling $h$, such that
 $U P_1 = P_2 U. $
 Then there exists a constant $\beta >0$, such that the equalities
\begin{align*}
n_1(\tau(w)) &= \beta n_2(w),\\
b_1(\tau(x),\tau(y)) &= \beta b_2(x,y),
\end{align*}
hold for all $w,x,y \in X_2$.
\end{theorem}

\begin{proof}
As $P_2 h = h$ if and only if $\mathcal{L}_2h = 0$,  the function $h$ is harmonic by Theorem~\ref{Gleichung_discrete}. The proof of Theorem \ref{Gleichung_discrete} shows that $UP_1 = P_2 U$ implies intertwining of the continuous time semigroups associated with $(b_i,0)$ over $(X_i,n_i), i=1,2$.  Therefore, Corollary \ref{stability_recurrence} shows that $(b_2,0)$ is also recurrent. Thus, the function $h$ is constant by Proposition~\ref{kernel Laplacian recurrent} and the statement follows from Theorem~\ref{Gleichung_discrete}.
\end{proof}

{\bf Remark.} In Section \ref{Recurrent} we defined recurrence via properties of continuous time Markovian semigroups and then observed that this definition is independent of the underlying measure $m$. However, this may seem artificial in the discrete time setting where one usually uses a different definition for recurrence. Namely, one requires the sum
$$\sum_{n\geq 1} P^n 1_x (y)$$
to diverge for all $x,y \in X$. It is well known that the divergence of this series can be characterized by positive superharmonic functions being constant (see e.g. \cite{Woe}). Thus, Proposition \ref{kernel Laplacian recurrent} shows that these two notions of recurrence agree.

\begin{appendix}

\section{A proposition on intertwining}
In this section we discuss  a characterization on intertwining. Such
a characterization is well-known in many contexts. It is certainly
also known in our context. As we could not find it in the
literature, we include a discussion.

\begin{proposition}\label{intertwining_operators} Let $B_1,B_2$ be Banach  spaces and  let $S_1(t)$, $t\geq 0$,  and $S_2 (t)$, $t\geq 0$,  strongly continuous semigroups with generators $-L_1$ and $-L_2$ respectively.   Let  $U\in\LL(B_1,B_2)$ invertible. Then, the following assertions are equivalent:
\begin{itemize}
\item[(i)]$U S_1 (t) =S_2 (t) U$ for all $t\geq 0$.
\item[(ii)]$UD(L_1)=D(L_2)$ and $UL_1 f=L_2 Uf$ for all $f\in D(L_1)$.
\end{itemize}
\end{proposition}
\begin{proof} We write $e^{-t L_i}$ for $S_i (t)$, $i =1,2$.

\smallskip

(i)$\implies$(ii): The operator $L_i$ is the generator of the semigroup $(e^{-tL_i})_{t\geq 0}$, $i\in\{1,2\}$, hence
\begin{align*}
D(L_i)&=\left\lbrace f\in B_i\,:\,\lim_{t\downarrow 0}\frac 1 t(f-e^{-tL_i}f)\text{ exists}\right\rbrace,\\
L_if&=\lim_{t\downarrow 0}\frac 1 t(f-e^{-tL_i}f).
\end{align*}
Thus, we have (since $U$ is a homeomorphism)
\begin{align*}
D(L_2)&=\left\lbrace f\in B_2\,:\,\lim_{t\downarrow 0}\frac 1 t(f-Ue^{-tL_1}U^{-1}f)\text{ exists}\right\rbrace\\
&=\left\lbrace f\in B_2\,:\,\lim_{t\downarrow 0}U\frac 1 t (U^{-1}f-e^{-tL_1}U^{-1}f)\text{ exists}\right\rbrace\\
&=\left\lbrace f\in B_2\,:\,\lim_{t\downarrow 0}\frac 1 t(U^{-1}f-e^{-tL_1}U^{-1}f)\text{ exists}\right\rbrace\\
&=UD(L_1)
\intertext{and for all $f\in B_1$}
L_2Uf&=\lim_{t\downarrow 0}\frac 1 t(Uf-e^{-tL_2}Uf)\\
&=U\lim_{t\downarrow 0}\frac 1 t(f-e^{-tL_1}f)\\
&=UL_1 f.
\end{align*}

(ii)$\implies$(i): Let be $S_t=Ue^{-t L_1}U^{-1}$. Then $(S_t)_{t\geq 0}$ is a strongly continuous semigroup on $B_2$
\begin{align*}
S_{t+s}&=Ue^{-(t+s)L_1}U^{-1}=Ue^{-t L_1}U^{-1}Ue^{-sL_1}U^{-1}=S_t S_s\\
\lim_{t\to 0}Ue^{-tL_1}U^{-1}f&=U\lim_{t\to 0}e^{-tL_1}U^{-1}f=UU^{-1}f=f
\end{align*}
for all $s,t\geq 0,\,f\in B_2$.
The generator $L$ of $(S_t)_{t\geq 0}$ is given by
\begin{align*}
D(L)&=\{f\in B_2 \, : \, \lim_{t\to 0}\frac 1 t(S_t f-f)\text{ exists}\}\\
&=\{f\in B_2 \, : \, \lim_{t\to 0}\frac 1 t(e^{-tL_1}U^{-1}f-U^{-1}f)\text{ exists}\}\\
&=\{f\in B_2 \, : \, U^{-1}f\in D(L_1)\}\\
&=U D(L_1)\\
&=D(L_2)
\end{align*}
and for all $f\in D(L)=D(L_2)$
\begin{align*}   
Lf&=\lim_{t\to 0}\frac 1 t(Ue^{-tL_1}U^{-1}f-f)\\
&=U\lim_{t\to 0}\frac 1 t(e^{-tL_1}U^{-1}f-U^{-1}f)\\
&=UL_1U^{-1}f\\
&=L_2f.
\end{align*}
Since the generator determines the semigroup, we have $Ue^{-t L_1}=e^{-t L_2}U$ for all $t\geq 0$.
\end{proof}

\end{appendix}

\end{document}